\documentclass[12pt]{amsart}

\usepackage{amsmath,amssymb,amsthm,fullpage}

\newtheorem{theorem}{Theorem}[section]

\newtheorem{corollary}[theorem]{Corollary}
\newtheorem{lemma}[theorem]{Lemma}

\newtheorem{proposition}[theorem]{Proposition}

\def\C{{\mathbb C}}

\def\N{{\mathbb N}}

\def\R{{\mathbb R}}
\def\Z{{\mathbb Z}}

\def\T{{\mathbb T}}

\def\supp{\textup{supp}}
\def\Res{\textup{Res}}

\begin{document}

\title[Heights of polynomials]{Heights of polynomials over lemniscates}

\author{Igor Pritsker}
\address{Department of Mathematics, Oklahoma State University, Stillwater, OK 74078, U.S.A.}
\email{igor@math.okstate.edu}

\date{}

\begin{abstract}
 We consider a family of heights defined by the $L_p$ norms of polynomials
with respect to the equilibrium measure of a lemniscate for $0 \le p \le \infty$, where
$p=0$ corresponds to the geometric mean (the generalized Mahler measure) and
$p=\infty$ corresponds to the standard supremum norm. This special choice of the
measure allows to find an explicit form for the geometric mean of a polynomial, and
estimate it via certain resultant. For lemniscates satisfying appropriate hypotheses, we establish explicit polynomials of
lowest height, and also show their uniqueness. We discuss relations between the standard
results on the Mahler measure and their analogues for lemniscates that include
generalizations of Kronecker's theorem on algebraic integers in the unit disk, as well as of
Lehmer's conjecture.
\end{abstract}



\dedicatory{Dedicated to the memory of  Peter Borwein}

\maketitle

\section{Polynomials of small height over lemniscates}

Let $\mu$ be a positive unit Borel measure with compact support $S=\supp\,\mu$ in the complex plane $\C.$ For any polynomial $P$,
one can define the standard $L_p(\mu)$ norms by setting
\begin{align}\label{Lp-norms}
  \|P\|_p := \left(\int |P|^p\,d\mu\right)^{1/p},\ 0<p<\infty, \quad\mbox{and}\quad \|P\|_\infty := \sup_S |P|,\ p=\infty.
\end{align}
The other endpoint $p=0$ of this range corresponds to the geometric mean with respect to the measure $\mu$:
\begin{align}\label{G-mean}
  \|P\|_0 = M(P) := \exp\left(\int \log|P|\,d\mu\right).
\end{align}
It is well known that these norms are subordinated as follows:
\begin{align}\label{sub}
  M(P) \le  \|P\|_p \le \|P\|_\infty,\quad 0<p<\infty,
\end{align}
where the first inequality is a consequence of Jensen's inequality \cite[pp. 138, 152]{HLP}, and the second one is immediate.

A particular choice of the measure $\mu$ in \eqref{Lp-norms} and \eqref{G-mean} certainly depends on the problem one wants to consider.
Perhaps the most classical and popular choice is the uniform distribution $d\mu(e^{i\theta}) = d\theta/(2\pi)$ on the unit
circumference $\T.$ This provides us with a very important tool in the study of algebraic numbers, namely with the Mahler measure.
For an arbitrary polynomial $P(z)=c_n\prod_{k=1}^{n}(z-z_k)$ with $c_n\neq 0,$ the Mahler measure is given by
\begin{align}\label{Mahler}
  M(P) := \exp\left(\frac{1}{2\pi}\int \log|P(e^{i\theta})|\,d\theta\right) = |c_n| \prod_{k=1}^{n} \max(1,|z_k|),
\end{align}
where the second equality is a well known consequence of Jensen's formula, see \cite{Bo} and \cite{EW} for background and applications.

The goal of this paper is to construct a natural family of polynomial heights over lemniscates that generalizes and includes the full range
of the above $L_p$ norms on the unit circle $\T.$ In many important cases, we solve the problem of finding integer polynomials of smallest
height over a lemniscate, and show their uniqueness. The generalized Mahler measure we introduce here is used to study algebraic integers on
and near lemniscates, and obtain results that directly correspond to Kronecker's theorem on algebraic integers in the unit disk \cite{Kr},
and other standard results related to the Mahler measure.

For a polynomial $V(z) = a_m \prod_{k=1}^m (z-\zeta_k) \in \C[z]$ with $a_m\neq 0$, and for any $r>0$, define its lemniscate as
\begin{align}\label{L}
L := \{z\in\C: |V(z)|=r\}.
\end{align}
We also use the filled-in lemniscate $E$, i.e., the union of lemniscate $L$ and its interior in $\C$:
\begin{align}\label{E}
E := \{z\in\C: |V(z)|\le r\}.
\end{align}
It is clear that if $r=0$ then both $L$ and $E$ degenerate to the set of zeros of $V$. If $r>0$ is sufficiently small, then $L$ consists of
analytic closed curves, with each curve surrounding a zero of $V$. In particular, if all zeros of $V$ are simple, then we have exactly
$m$ analytic non-intersecting components of $L$. When $r$ increases, those components monotonically increase in size because of the
Maximum Modulus Principle for $V$, and they eventually meet at the critical points of $V$. For sufficiently large $r$, the lemniscate $L$
is always a single analytic curve that approaches infinity as $r\to\infty.$ More details on the properties and geometry of lemniscates can
be found in \cite[pp. 17--22]{Wa}. Some of the most historically important lemniscates include Cassini ovals and the lemniscate of Bernoulli,
where both examples are defined by $V(z)=z^2-1$ with respective values $r\in(0,1)$ and $r=1.$ While lemniscates are rather special
curves, they can approximate essentially arbitrary geometric shapes by Hilbert's Lemniscate Theorem, see \cite[p. 158]{Ra}.

We now specify the measure $\mu$ in the definitions for the polynomial heights \eqref{Lp-norms} and \eqref{G-mean} as the equilibrium
measure of the lemniscate $L$ in the sense of logarithmic potential theory (see \cite{Ra} and  \cite{Ts}), which is a unit measure supported on $L$ expressing
the steady state distribution of charge if $L$ is viewed as conductor. This equilibrium measure is defined for general compact sets in potential
theory and complex analysis, and it already found many applications in number theory, e.g., in connection with asymptotic distribution of
algebraic numbers. For example, the equilibrium measure was used to define generalizations of the Weil height in \cite{Ru} and \cite{FR},
and of the Mahler measure in \cite{PrCr}. It is known that in our case the equilibrium measure of the lemniscate $L$ of \eqref{L} is given explicitly by
\begin{align}\label{mu}
  d\mu(z) = \frac{|V'(z)|}{2\pi m r}\,|dz|, \quad z\in L,
\end{align}
see \cite[p. 350]{To}. 
Thus the support of $\mu$ is $L$, and the heights we use in this paper
are defined by the integrals with respect to this measure in \eqref{Lp-norms} and \eqref{G-mean}. Note that the supremum norm in \eqref{Lp-norms}
is actually the maximum of $|P|$ over $L.$ It is immediately clear that for $V(z)=z$ and $r=1$ the above
measure reduces to the familiar $d\theta/(2\pi)$ on $\T.$ In fact, one does not need to know potential theory to understand the results
of this paper, which is primarily due to the direct generalization of \eqref{Mahler} in the explicit version of Jensen’s formula for lemniscates stated below.

\begin{proposition} \label{GMf}
Let $L$ be the lemniscate $\{z\in\C:|V (z)| = r\}$ defined by \eqref{L}. If $P(z)=c_n\prod_{k=1}^{n}(z-z_k) \in \C[z]$ is any polynomial
with $c_n\neq 0,$ then the generalized Mahler measure $M_L(P)$ with respect to the equilibrium measure $\mu$ of $L$ in \eqref{mu} is given by
\begin{align}\label{GM}
  M_L(P) = |c_n|\,|a_m|^{-n/m} \left(\prod_{k=1}^{n} \max(r,|V(z_k)|)\right)^{1/m}.
\end{align}
Furthermore, we have
\begin{align}\label{Res}
  M_L(P) \ge |a_m|^{-n/m}\,|\Res(P,V)|^{1/m},
\end{align}
where $\Res(P,V)$ is the resultant of $P$ and $V$.
\end{proposition}

We observe once again that if $V(z)=z$ and $r=1$, then $L=\T$ and \eqref{GM} reduces to the regular Mahler measure \eqref{Mahler}.
It is also transparent from either \eqref{G-mean} or \eqref{GM} that the generalized version of the Mahler measure inherits many properties
of the classical one. For example, it is multiplicative, i.e., $M_L(PQ)=M_L(P)M_L(Q)$ for any pair of polynomials $P,Q.$ In all subsequent results,
we consider polynomials that are not identically zero, but one can also set $M_L(P)=1$ for $P\equiv 0$, following the standard agreement for the classical Mahler measure.

Problems on minimizing certain norms or heights by polynomials are abundant in mathematics in general and in number theory in particular.
A selection of extremal problems on polynomials with integer coefficients can be found, for example, in \cite{Bo}. We are primarily interested
in the problem of minimizing the heights \eqref{Lp-norms} and \eqref{G-mean}, defined by the measure $\mu$ of \eqref{mu} on lemniscates,
for polynomials with integer coefficients. This group of problems has a long history, going back to the work of Hilbert \cite{Hi} on
integer polynomials with small $L_2$ norm on an interval of the real line, some results of Schur \cite{Sch} and Fekete \cite{Fe} for the supremum
norm on general sets of the real line and the complex plane, and numerous applications, see, e.g., \cite{Bo}, \cite{Mo}, \cite{Fer}, \cite{Tr}, and references
therein. Despite this long history and extensive work, finding exact solutions and extremal polynomials is usually out of reach, and even providing
good estimates proved difficult. Nevertheless, we give the exact values for the lowest heights and exhibit the smallest polynomials, under the hypothesis that $V \in \Z[z]$ is irreducible, on many important classes of lemniscates.

\begin{theorem} \label{MinH}
  Let $L$ be defined by \eqref{L}, where $V(z) = a_m z^m + \ldots \in\Z[z]$ is irreducible and $0<r \le 1/|a_m|$.
  The smallest value of heights \eqref{Lp-norms} and \eqref{G-mean}, with respect to $\mu$ of \eqref{mu},
  for the family of all polynomials $P\in\Z[z]$ of degree  $\deg(P) \le km,\ k\in\N,$ such that $P\not\equiv 0$, is given by
\begin{align}\label{Lpmin}
  \inf\{\|P\|_p: P\in\Z[z],\ P\not\equiv 0,\ \deg(P) \le km\} = r^k, \ 0<p\le\infty,
\end{align}
and
\begin{align}\label{GMmin}
  \inf\{M_L(P): P\in\Z[z],\ P\not\equiv 0,\ \deg(P) \le km\} = r^k.
\end{align}

Furthermore, we have the following uniqueness results:\\
\textup{(i)} If $r<1/|a_m|$ then equality holds in \eqref{Lpmin} or \eqref{GMmin} if and only if $P(z)=\pm V^k(z)$;\\
\textup{(ii)} If $r=1/|a_m|<1$, then equality holds in \eqref{Lpmin} if and only if $P(z)=\pm V^k(z)$;\\
\textup{(iii)} For $|a_m|=r=1$, equality holds in \eqref{Lpmin} if and only if $P(z)=\pm V^k(z)$ or $P(z)=\pm 1$.
\end{theorem}

In the special case $V(z)=z$, this result states that the extremal polynomials for circles $L$ centered at the origin are
the monomials, which was certainly known before. For general lemniscates, Theorem \ref{MinH} is completely new for all $p\in[0,\infty)$.
If $p=\infty$ then \eqref{Lpmin} was already stated in Theorem 1.5 of \cite{Pr}, but the uniqueness of extremal polynomials part is new even
in this case. From the number theoretic point of view, the most interesting cases of Theorem \ref{MinH} are $p=\infty$ that corresponds to
the integer Chebyshev (or integer transfinite diameter) problem, and $p=0$ that corresponds to the generalized Mahler measure. A selection
of references on the integer Chebyshev problem include \cite{Am90}, \cite{Am93}, \cite{Bo}, \cite{BE}, \cite{BPP}, \cite{Fl}, \cite{FRS},
\cite{Mo}, \cite{Pr}, but we do not attempt to give a comprehensive up-to-date survey. A more or less general discussion of the integer
Chebyshev problem is contained in \cite{Pr}. The most studied case of the integer Chebyshev problem is related to the interval $[0,1]\subset\R,$
however even this instance of the problem remains open to a large extent. As far as
we know, the only non-trivial case when the integer Chebyshev problem is explicitly solved (by finding the smallest value of the supremum norm on
the corresponding set and exhibiting the extremal integer polynomials) is contained in Theorem \ref{MinH}. It is interesting that in the latter case
the polynomials of the smallest supremum norm are unique (up to a change of sign, of course). This is not true for the integer Chebyshev problem
on $[0,1]$, see \cite[p. 665]{BE} for an example of two different extremal polynomials of degree 4. While Theorem \ref{MinH} does not give uniqueness results for the polynomials $P$ minimizing $M_L(P)$ when $r = |a_m| = 1$ in (ii) and (iii), we fill this gap below in Theorem \ref{GKr}. It is certainly of interest to study the remaining cases that are not covered by Theorem \ref{MinH} and other results of this paper, e.g., when $r = |a_m| > 1$.

For $r/|a_m| > 1$, the problem can be handled by purely analytic methods, and we find that the smallest height polynomials on a lemniscate defined by any polynomial $V$ with complex coefficients are $P(z)=\pm 1$, as stated below in Proposition \ref{Llarge}. One can easily see from \eqref{GM} that
\begin{align}\label{GMlb}
  M_L(P) \ge |c_n|\,\left(\frac{r}{|a_m|}\right)^{n/m} \ge \left(\frac{r}{|a_m|}\right)^{n/m},
\end{align}
where $P(z)=c_n z^n + \ldots\in\Z[z],\ c_n\neq 0.$ Combining this with \eqref{sub}, we obtain the following result:

\begin{proposition} \label{Llarge}
  Let $L$ be defined by \eqref{L}, where $V(z) = a_m z^m + \ldots \in\C[z]$ and $r/|a_m| \ge 1$.
  The smallest value of heights \eqref{Lp-norms} and \eqref{G-mean}, with respect to $\mu$ of \eqref{mu}, for the family of all polynomials $P\in\Z[z]$ such that $P\not\equiv 0$, is given by
\begin{align}\label{LpGMmin}
  \inf M_L(P) = \inf \|P\|_p = 1.
\end{align}
For $r/|a_m| > 1$, equality holds in \eqref{LpGMmin} if and only if $P(z)=\pm 1$.
\end{proposition}

We give a detailed analysis for the polynomials of lowest height is the sense of the generalized Mahler measure in the next section. Section \ref{Proofs}
contains proofs of all results.

\section{Generalized Mahler measure and algebraic integers}

The importance of the Mahler measure \eqref{Mahler} in number theory is primarily related to its role in the study of algebraic numbers on and near the unit
circle $\T$, see \cite{Bo}, \cite{EW}, \cite{Sm1}, \cite{Sm2}, etc. We show in this section that the generalized Mahler measure can be used in the same
fashion to understand the distribution of algebraic numbers on or near lemniscates.

It is immediate from \eqref{Mahler} that the (classical) Mahler measure satisfies $M(P)\ge 1$ for any $P\in\Z[z],\ P\not\equiv 0,$
with equality holding if and only if $\pm P$ is monic
and has all roots in the closed unit disk. Since the Mahler measure is multiplicative, we obtain that $M(P)=1$ iff $M(Q)=1$ for every irreducible
factor of $P$. Hence we can restrict our attention to irreducible polynomials. Kronecker's theorem \cite{Kr} now gives for a monic irreducible $P\in\Z[z]$ that $M(P)=1$  if and
only if either $P$ is cyclotomic or $P(z)=z.$ In particular, the polynomials of smallest Mahler measure are not unique in this case. Note that Theorem \ref{MinH} does not
provide uniqueness information as $V(z)=z$ with $a_m=r=1$ here, but we fill this gap below.

Suppose that the polynomial $V$ defining the lemniscate $L$ in \eqref{L} is monic, but not necessarily irreducible. Following the same argument as above in the generalized Mahler measure case, we observe from \eqref{GM} and \eqref{GMlb} that $M_L(P) \ge r^{n/m}$ for any $P\in\Z[z],\ P\not\equiv 0,\ \deg(P)=n.$ Moreover, $M_L(P) = r^{n/m}$ if and only if $\pm P$ is a monic polynomial with all roots in the filled-in lemniscate $E$ defined by \eqref{E}. Since the generalized Mahler measure is multiplicative, we can restrict the argument to
irreducible polynomials, as before. The following version of Kronecker's theorem holds for $M_L(P)$ on the lemniscate $L$ of a polynomial $V$, which is stated in terms of the composition $\Phi\circ V$ for a cyclotomic polynomial $\Phi.$

\begin{theorem} \label{GKr}
Let $L$ be defined by \eqref{L}, where $V(z) = z^m + \ldots \in\Z[z]$ is monic, and $r = 1$.
The generalized Mahler measure \eqref{GM} satisfies
\begin{align}\label{GMmon}
  M_L(P) \ge 1,\quad P\in\Z[z],\ P\not\equiv 0.
\end{align}
Equality is attained above if and only if $P$ has leading coefficient $\pm 1$, and all roots of $P$ are located in $E$ defined by \eqref{E}.

More precisely, we have $M_L(P)=1$ for a monic irreducible $P\in\Z[z]$ if and only if either $P\mid V$ or $P\mid \Phi\circ V$ for a cyclotomic polynomial $\Phi.$ Hence if $\alpha$ is an algebraic integer contained in $E$ together with all of its conjugates, then we have two possibilities:\\
\textup{(i)} $\alpha$ is a root of $V$ when $\alpha\in E^\circ$;\\
\textup{(ii)} $\alpha$ is a root of $\Phi\circ V$, for a cyclotomic polynomial $\Phi,$ when $\alpha\in L$.
\end{theorem}

One interesting consequence of this result is the description of all complete sets of conjugate algebraic integers located on a lemniscate satisfying conditions of Theorem \ref{GKr}. The fact that we have infinitely many complete sets of conjugate algebraic integers was already observed by Fekete and Szeg\H{o} in \cite[p. 162]{FS},
in the form of roots for the equation $V^n(z)=1,\ n\in\N,$ but they did not give a full description of such sets on a lemniscate.

\begin{corollary} \label{AlgInt}
  Let $L$ be defined by \eqref{L}, where $V(z) = z^m + \ldots \in\Z[z]$ and $0< r \le 1$.\\
\textup{(i)} If $0<r<1$ then there are no complete sets of conjugate algebraic integers located on $L$;\\
\textup{(ii)} For $r=1$, $\{\alpha_k\}_{k=1}^n\subset L$ is a complete set of conjugate algebraic integers if and only if $\{\alpha_k\}_{k=1}^n$
is a (complete conjugate) subset of the roots of $\Phi\circ V$, where $\Phi$ is a cyclotomic polynomial.
\end{corollary}

The problem of describing complete sets of conjugate algebraic integers located in a given set is another old question that remains open in many basic cases. For example, it is known from the same work of Kronecker \cite{Kr} that complete sets of conjugates in an interval $[a,a+4],$ with $a\in\Z$, are roots of the scaled Chebyshev polynomial $2T_n((x-a)/2)$. However, if $a\not\in\Z$ then the description of complete sets of conjugates is not known, nor it is known whether there are infinitely many such sets in $[a,a+4]$. For more information on this problem see \cite{Ro62}, \cite{Ro64}, \cite{Ro69} and the survey \cite{Mc}, where one can find many additional references.

\medskip
Our next logical step is an analogue of the celebrated Lehmer's conjecture. The original version appeared in \cite{Le}, and it is very well known and thoroughly studied, see \cite{Bo}, \cite{EW}, \cite{Sm1}, \cite{Sm2} for references. Since $M(P) \ge 1$ for any $P\in\Z[z],\ P\not\equiv 0,$ Lehmer was interested in a natural question
on how close the Mahler measure of a non-cyclotomic polynomial can be to $1$. In fact, he observed from computations that the smallest such measure seems to be
achieved by the polynomial $\mathcal{L}(z) =  z^{10} + z^9 - z^7 - z^6 - z^5 - z^4 - z^3 + z + 1$:
\[
M(P) \ge M(\mathcal{L}) \approx 1.176280818,
\]
where $P\in\Z[z]$ and $M(P)\neq 1.$ We do not attempt to prove the original Lehmer's conjecture, but instead relate it to an analogous statement for lemniscates. Theorem
\ref{GKr} gives that $M_L(P) \ge 1,\ P\in\Z[z],\ P\not\equiv 0,$ essentially as in the classical case $L=\T$, and also explains when the value 1 is achieved. Thus we are led to
the question on the greatest lower bound for the generalized Mahler measure of all polynomials $P$ with $M_L(P)>1$:
\begin{align}\label{GLehmer}
  B_L := \inf\{M_L(P): P\in\Z[z],\ M_L(P)>1\},
\end{align}
where $L$ is defined by \eqref{L}, with $V(z) = z^m + \ldots \in\Z[z]$ and $r = 1$. One can rephrase the original Lehmer's question as whether $B_\T = M(\mathcal{L}),$
or, in a weaker form, whether $B_\T>1$ (with $V(z)=z$). We show that $B_\T>1$ is equivalent to $B_L>1,$ i.e., the weaker form of Lehmer's conjecture holds if and only if its analogue holds
for any lemniscate.

\begin{theorem} \label{L-GL}
  Let $L$ be defined by \eqref{L}, with $V(z) = z^m + \ldots \in\Z[z]$ and $r = 1$. Then
\begin{align}\label{LGL}
(B_\T)^{1/m} \le  B_L \le B_\T.
\end{align}
\end{theorem}

Thus $B_\T>1$ implies $B_L>1$ for all lemniscates $L$ satisfying the above assumptions, and $B_L>1$ for a single such lemniscate implies $B_\T>1$. It is clear from
the proof of Theorem \ref{L-GL} that most of known lower bounds for the Mahler measure can be translated into the corresponding analogues for the generalized Mahler
measure.

\section{Proofs} \label{Proofs}

\begin{proof}[Proof of Proposition \ref{GMf}]
We obtain from \eqref{G-mean} that
\begin{align}\label{GMcalc}
  \log M_L(P) &= \int \log|P|\,d\mu = \log|c_n| + \sum_{k=1}^{n} \int \log|t-z_k|\,d\mu(t),
\end{align}
where we used that the equilibrium measure $\mu$ defined in \eqref{mu} is a unit measure. We recall some
well known properties of the equilibrium potential $\int \log|t-z|\,d\mu(t),$ such as \cite[p. 59]{Ra}
\[
\int \log|t-z|\,d\mu(t) = \log \textup{cap}(L),\quad z\in E,
\]
where $\textup{cap}(L)$ is the logarithmic capacity of $L$ known explicitly \cite[pp. 134--135]{Ra} as
\[
\textup{cap}(L) = \left(\frac{r}{|a_m|}\right)^{1/m}.
\]
Another important fact we need is the connection between the equilibrium potential and the Green function $g(z,\infty)$ for the complement of $E$ with logarithmic pole at infinity \cite[p. 107]{Ra}:
\[
\int \log|z-t|\,d\mu(t) = g(z,\infty) + \log \textup{cap}(L),\quad z\in\C\setminus  E.
\]
Since the Green function is also known explicitly in this case \cite[p. 134]{Ra} as
\[
g(z,\infty) = \frac{1}{m} \log\frac{|V(z)|}{r},\quad z\in\C\setminus  E,
\]
we can summarize these findings in the following formula for the equilibrium potential:
\begin{align*}
\int \log|z-t|\,d\mu(t) &= \left\{
                                   \begin{array}{ll}
                                     \frac{1}{m} \log\frac{r}{|a_m|}, & \hbox{$z\in E$;} \\
                                     \frac{1}{m} \log\frac{|V(z)|}{|a_m|}, & \hbox{$z\in\C\setminus E$.}
                                   \end{array}
                                 \right. \\
&= \frac{1}{m} \log\frac{\max(r,|V(z)|)}{|a_m|}, \quad z\in\C.
\end{align*}
Applying the latter explicit evaluation of the equilibrium potential in \eqref{GMcalc}, and passing to the exponential form, we verify \eqref{GM}.

It is immediate from \eqref{GM} that
\begin{align*}
  M_L(P) &= |c_n|\,|a_m|^{-n/m} \left(\prod_{k=1}^{n} \max(r,|V(z_k)|)\right)^{1/m} \\
&\ge |a_m|^{-n/m} \left( |c_n|^m \prod_{k=1}^{n} |V(z_k)|)\right)^{1/m} = |a_m|^{-n/m}\,|\Res(P,V)|^{1/m},
\end{align*}
so that \eqref{Res} is also proved.
\end{proof}

We state a simple fact needed in the proof of  Theorem \ref{MinH}.

\begin{lemma} \label{Reslem}
Let $L$ be defined by \eqref{L} with $V(z)=a_m z^m + \ldots \in \Z[z],\ a_m\neq 0.$\\
If $Q\in\Z[z],\ \deg(Q)=l\in\N,$ satisfies
\[
|a_m|^l (M_L(Q))^m < 1,
\]
then $\Res(Q,V)=0$. If, in addition, $V$ is irreducible, then $V\mid Q.$
\end{lemma}

\begin{proof}
It follows from \eqref{Res} that
\begin{align*}
  |\Res(Q,V)| \le |a_m|^l (M_L(Q))^m < 1.
\end{align*}
Since $Q,V\in\Z[z]$, we have that $\Res(Q,V)\in\Z,$ and therefore $\Res(Q,V)=0.$ The last claim is clear.
\end{proof}

\begin{proof}[Proof of Theorem \ref{MinH}]
Note that if $P=V^k,\ k\in\N,$ then $\|P\|_p=M_L(P)=r^k.$  Thus to prove \eqref{Lpmin} and \eqref{GMmin} we need to show that $\|P\|_p<r^k$ or
$M_L(P)<r^k$ is not possible for $P\in\Z[z],\ P\not\equiv 0,\ \deg(P)=n \le km.$ If $\|P\|_p<r^k$ then $M_L(P)<r^k$ by \eqref{sub}, so that we assume
the latter inequality holds and reach a contradiction. For that purpose, suppose that $P=V^d R$, where $d\ge 0$ and $R\in\Z[z],\ V\nmid R,\ \deg(R) = n-dm \le (k-d)m.$
Observe that $d<k$ here, because otherwise $P=cV^k,\ |c|\in\N,$ and $\|P\|_p=M_L(P)=|c|r^k\ge r^k$. Since $M_L(P) = M_L(V^d) M_L(R) = r^d M_L(R)$
and $M_L(P)<r^k$ by our assumption, we conclude that $M_L(R)<r^{k-d}.$ But $|a_m| r \le 1$ gives that
\[
|a_m|^{n-dm} (M_L(R))^m \le |a_m|^{(k-d)m} (M_L(R))^m < (|a_m| r)^{(k-d)m} \le 1.
\]
It follows from Lemma \ref{Reslem} with $Q=R$ that $V\mid R$, contradicting our assumption.

We now turn to the proof of uniqueness statements (i)-(iii).

\textit{Proof of (i)} Note first that $\|P\|_p=r^k$ implies $M_L(P)=r^k$ by \eqref{sub} and \eqref{GMmin}. If $|a_m| r < 1$ and $M_L(P)=r^k$,
then we let $P=V^d R$, where $R\in\Z[z],\ V\nmid R,\ \deg(R) = n-dm \le (k-d)m,$ as in the first part of this proof. This yields $M_L(P) = M_L(V^d) M_L(R) = r^d M_L(R)$
and $M_L(R)=r^{k-d}.$ If $d<k$ then $|a_m| r < 1$ gives that
\[
|a_m|^{n-dm} (M_L(R))^m \le |a_m|^{(k-d)m} (M_L(R))^m = (|a_m| r)^{(k-d)m} < 1,
\]
and Lemma \ref{Reslem} leads us to the contradiction $V\mid R$. Hence $d=k$ and $R$ is a constant polynomial, which must be $\pm 1$ to satisfy $M_L(P)=r^k$.

\textit{Proof of (ii)} Here we have $|a_m| r = 1,\ r<1$ and $\|P\|_p=r^k,\ 0<p\le\infty.$ If $M_L(P) < \|P\|_p = r^k$ then we proceed in essentially the same way as
before by assuming that $P=V^d R$, where $R\in\Z[z],\ V\nmid R,\ \deg(R) = n-dm \le (k-d)m.$ Since $M_L(P) = M_L(V^d) M_L(R) = r^d M_L(R)$
and $M_L(P)<r^k,$ we obtain that $M_L(R)<r^{k-d}$ and $d<k.$ This gives again that
\[
|a_m|^{n-dm} (M_L(R))^m \le |a_m|^{(k-d)m} (M_L(R))^m < (|a_m| r)^{(k-d)m} = 1.
\]
Applying Lemma \ref{Reslem}, we obtain that $V\mid R$, contradicting our assumption.

It remains to handle the case when $M_L(P)=\|P\|_p$. From the equality case in Jensen's inequality, see \cite[p. 138]{HLP},
we know that $M_L(P)=\|P\|_p$ is possible if and only if $|P(z)|$ is constant for all $z\in L.$
The latter means that $|P(z)|=r^k,\ z\in L,$ and $P$ is not a constant, as $r<1.$ Once again, we set $P=V^d R$, where $R\in\Z[z],\ V\nmid R,\ \deg(R) = n-dm \le (k-d)m,$
which implies $|R(z)|=r^{k-d},\ z\in L.$ Since the roots of $V$ denoted by $\{\zeta_j\}_{j=1}^m$ are all located inside $L$,
the Maximum Modulus Principle implies that $|R(\zeta_j)| < r^{k-d},\ j=1,\ldots,m.$ We obtain directly from the resultant formula that
\[
|\Res(R,V)| = |a_m|^{n-dm} \prod_{j=1}^{m} |R(\zeta_j)| < |a_m|^{(k-d)m} r^{(k-d)m} = 1.
\]
Thus $\Res(R,V)=0$, and the polynomials $V$ and $R$ have a common root. Since $V$ is irreducible, we obtain that
$V\mid R$ in contradiction to our assumption.

\textit{Proof of (iii)} This proof is almost identical to that of case (ii), with one additional possibility that an extremal polynomial $P$ may be constant.
But if $P\equiv c$ then $c=\pm 1$, as $\|P\|_p = r^k = 1.$
\end{proof}

\begin{proof}[Proof of Proposition \ref{Llarge}]
As we already noted before the statement of this proposition, a combination of \eqref{sub} and \eqref{GMlb} gives \eqref{LpGMmin}. Indeed,
\[
\|P\|_p \ge M_L(P) \ge |c_n|\,\left(\frac{r}{|a_m|}\right)^{n/m} \ge \left(\frac{r}{|a_m|}\right)^{n/m} \ge 1,
\]
where $P(z)=c_n z^n + \ldots\in\Z[z],\ c_n\neq 0.$ Equality is always achieved in \eqref{LpGMmin} by $P=\pm 1.$ But if $n>0$ and $r/|a_m|>1$, then the above estimate shows that such a polynomial $P$ cannot be extremal for \eqref{LpGMmin}. Thus any extremal $P$ must be constant when $r/|a_m|>1$, and obviously $P=\pm 1.$
\end{proof}

\begin{proof}[Proof of Theorem \ref{GKr}]
Inequality \eqref{GMmon} follows from \eqref{GM} and  \eqref{GMlb} as was already explained before the statement of  this theorem.
In the case of equality in \eqref{GMmon}, we have under our assumptions that for $P(z)=c_n\prod_{k=1}^{n}(z-z_k) \in \Z[z]$
with $c_n\neq 0$ the generalized Mahler measure
\begin{align*}
  M_L(P) = |c_n| \left(\prod_{k=1}^{n} \max(1,|V(z_k)|)\right)^{1/m} = 1.
\end{align*}
The above equality holds if and only if $|c_n|=1$ and $|V(z_k)|\le 1,\ k=1,\ldots,n,$ which is equivalent to $\{z_k\}_{k=1}^n \subset E.$

Suppose now that $P(z) = \prod_{k=1}^{n}(z-\alpha_k) \in \Z[z]$ is irreducible, i.e., it is the minimal polynomial for a complete set of conjugate
algebraic integers $\{\alpha_k\}_{k=1}^n$. We showed that $M_L(P)=1$ if and only if $\{\alpha_k\}_{k=1}^n \subset E$, which is the same as
$|V(\alpha_k)|\le 1,\ k=1,\ldots,n.$ Define the monic polynomial $Q(w) = \prod_{k=1}^{n}(w-V(\alpha_k))$, and note that its coefficients are
symmetric polynomials in $\{\alpha_k\}_{k=1}^n$ with integer coefficients. Hence the coefficients of $Q$ are integer polynomials in elementary
symmetric functions of $\{\alpha_k\}_{k=1}^n$, and are ultimately integers. Thus $Q$ is a monic polynomial with integer coefficients and roots
$\{V(\alpha_k)\}_{k=1}^n$ located in the closed unit disk. Kronecker's theorem now implies that either $V(\alpha_k)=0$ (when $|V(\alpha_k)|<1$)
or $V(\alpha_k)$ is a root of unity (when $|V(\alpha_k)|=1$). In the first case $V(\alpha_k)=0$, we have that $\alpha_k$ is also a root of the irreducible monic
polynomial $P$, so that $P\mid V$ and $V(\alpha_k)=0$ for all $k=1,\ldots,n.$ In the second case $|V(\alpha_k)|=1$, there is a cyclotomic polynomial $\Phi$ such that
$\Phi(V(\alpha_k))=0,$ which means that $\Phi\circ V\in\Z[z]$ shares a root with the irreducible polynomial $P\in\Z[z]$. Hence we obtain that $P\mid \Phi\circ V$ and
$\Phi(V(\alpha_k))=0$ for all $k=1,\ldots,n.$ This completes the proof of Theorem \ref{GKr}, and also of part (ii) for Corollary \ref{AlgInt}.
\end{proof}

\begin{proof}[Proof of Corollary \ref{AlgInt}]
We start with part (i). If a complete set of conjugate algebraic integers $\{\alpha_k\}_{k=1}^n$ is located on $L$ then
\begin{align*}
  M_L(P) = \left(\prod_{k=1}^{n} \max(r,|V(\alpha_k)|)\right)^{1/m} = r^{n/m}
\end{align*}
for the minimal polynomial
$P(z) = \prod_{k=1}^{n}(z-\alpha_k).$ Since $V(z)$ is assumed monic, we obtain that
\[
|a_m|^n (M_L(P))^m = r^n < 1,
\]
and applying Lemma \ref{Reslem}, we conclude that $\Res(P,V)=0$ and  $P\mid V.$ It follows that all roots of $P$ coincide with some roots of $V$, and are located inside $L$.

Part (ii) was already established in the proof of Theorem \ref{GKr}.
\end{proof}

\begin{proof}[Proof of Theorem \ref{L-GL}]
If $Q(w)=b_n \prod_{k=1}^{n}(w-\beta_k) \in\Z[w]$ is any polynomial such that $M(Q)>1$, then we either have $|b_n|>1$ or $|b_n|=1$ and
$|\beta_j|>1$ for some $j,\ 1\le j\le n.$ Consider $P(z)=Q(V(z))\in\Z[z],\ \deg(P)=mn,$ with roots $\{\alpha_{i,k}\},\ i=1,\ldots,m,\ k=1,\ldots,n,$
that satisfy $V(\alpha_{i,k})=\beta_k,\ i=1,\ldots,m,$ for all $k=1,\ldots,n.$ Note that $|\beta_j|>1$ implies that $\alpha_{i,j} \not\in E$ for all
$i=1,\ldots,m.$ Thus we have
\begin{align*}
M_L(P) &= |b_n| \left(\prod_{k=1}^{n} \prod_{i=1}^{m}\max(1,|V(\alpha_{i,k})|)\right)^{1/m} \\
&= |b_n| \prod_{k=1}^{n} \max(1,|\beta_k|) = M(Q).
\end{align*}
It is clear now that $M_L(P)>1$, and so $M_L(P)\ge B_L$ by \eqref{GLehmer}. Hence $M(Q)\ge B_L$ and we obtain $B_\T \ge B_L$ by taking infimum over all $Q$.

Let $P(z) = c_n \prod_{k=1}^{n}(z-\alpha_k) \in \Z[z]$ be any polynomial such that $M_L(P)>1,$ which means that either $|c_n|>1$ or
$|c_n|=1$ and $\alpha_j \not\in E$ for some $j,\ 1\le j\le n,$ by Theorem \ref{GKr}. Consider $Q(w) = c_n^m \prod_{k=1}^{n}(w-V(\alpha_k))$ and
observe that the coefficients of $Q/c_n^m$ are symmetric polynomials in $\{\alpha_k\}_{k=1}^n$ with integer coefficients, whose degree in each
$\alpha_k$ does not exceed $m$. Hence these coefficients are integer polynomials in elementary symmetric functions of $\{\alpha_k\}_{k=1}^n$,
of degree at most $m$, and therefore they become integers after multiplying by $c_n^m.$ Thus $Q\in\Z[w].$ We also note that either $|c_n|>1$ or $|V(\alpha_j)|>1$
implies $M(Q)>1.$ It follows that
\[
M_L(P) = |c_n| \left(\prod_{k=1}^{n} \max(1,|V(\alpha_k)|)\right)^{1/m} = \left(M(Q)\right)^{1/m} \ge (B_\T)^{1/m} \ge 1.
\]
Thus we obtain $B_L \ge (B_\T)^{1/m}$ from the definition \eqref{GLehmer}.
\end{proof}

\subsection*{Acknowledgements}
This research was partially supported by NSF via the American Institute of
Mathematics, and by the Vaughn Foundation endowed Professorship in Number Theory.

\end{document}